\documentclass{article}
\usepackage{fullpage}
\usepackage{graphicx} 
\usepackage{amsfonts}
\usepackage{amsthm}
\usepackage{amsmath}
\usepackage{amssymb}
\usepackage{mathabx}
\usepackage{mathrsfs}
\usepackage{mathptmx}
\usepackage{comment}
\usepackage{tikz}
\usetikzlibrary{arrows.meta}
\usepackage{enumitem}
\usepackage{hyperref}
\usepackage[capitalise]{cleveref}
\crefname{enumi}{\unskip}{\unskip}
\usepackage{xcolor}
\usepackage[
backend=biber,
style=alphabetic,
sorting=nyt,
maxnames=99,
doi=false,
url=false,
isbn=false
]{biblatex}
\usepackage{titling}

\usepackage[a4paper,left=2.5cm,right=2.5cm, footskip=30pt]{geometry}

\newcommand{\Perm}{\mathop{\mathrm{Perm}}}
\newcommand{\Sym}{\mathop{\mathrm{Sym}}}
\newcommand{\ds}{\mathop{d_\square}}

\newcommand{\supp}{\mathop{\mathrm{supp}}}

\newtheorem{theorem}{Theorem}[section]
\newtheorem{lemma}[theorem]{Lemma}

\theoremstyle{definition}
\newtheorem{claim}[theorem]{Claim}
\newtheorem{remark}[theorem]{Remark}
\newtheorem{question}[theorem]{Question}

\title{On the approximation of permutons}
\thanksmarkseries{alph}
\author{Bal\'azs Maga\thanks{HUN-REN Alfréd Rényi Institute of Mathematics, Budapest, Hungary. Email: \texttt{magab@renyi.hu}
\newline Supported by the NKFIH 152535 project, funded by the Ministry of Innovation and Technology of Hungary from the National Research, Development and Innovation Fund.}}
\date{}

\begin{document}

\maketitle
\begin{abstract}
We study the optimal rectangular-discrepancy approximation of permutons by finite permutations. We transfer bounds from discrepancy theory to this more restricted setup. Moreover, we show that superlinear approximation can occur only for permutons supported by graphs of measure-preserving functions, and demonstrate how the local regularity of this function obstructs approximability. We also consider the biased Brownian separable permuton and prove a lower bound on its approximation error by showing that its supporting measure-preserving function has Lipschitz points almost surely.

\end{abstract}

\section{Introduction}

A \emph{permuton} is a probability measure on the unit square with uniform marginals. Permutons are the natural limit objects of permutations \cite{HOPPEN201393}, where convergence is defined via convergent substructure densities arising from sampling (cf. graphons and graph convergence). Permutons have received much attention in recent years, see for example \cite{BASSINO2022108513, BorgaGwynneSun2025PermutonsMeandersLQG, Bouvel_Nicaud_Pivoteau_2026}.

We denote the set of permutons by $\Perm$. The convergence of a sequence of permutations $\pi_n\in \Sym(n)$ to $\mu\in\Perm$ is defined using pattern densities. The density $t(\sigma, \pi)$ of a pattern $\sigma\in \Sym(k)$ in $\pi\in \Sym(n)$ for $k\leq n$ is defined as the probability of $(\pi(i_j))_{j=1}^{k}$ being order-isomorphic to $(\sigma(j))_{j=1}^{k}$ for a uniform random increasing $k$-tuple $(i_j)_{j=1}^{k}$ in $[n]=\{1, 2, \dots, n\}$. For example, in the permutation $15342\in \Sym(5)$, five of the ten increasing 3-tuples  are order-isomorphic to 132 (153, 154, 152, 132, and 142), hence $t(132, 15342)=5/10.$

Analogously, the density $t(\sigma, \mu)$ of $\sigma$ in $\mu\in \Perm$ is defined as follows: taking an i.i.d. $k$-sample $(x_j, y_j)_{j=1}^{k}$, if $(i_j)_{j=1}^{k}$ is a permutation of $[k]$ such that $(x_{i_j})_{j=1}^{k}$ is an increasing sequence, $t(\sigma, \mu)$ is the probability of $(y_{i_j})_{j=1}^{k}$ being order-isomorphic to $(\sigma(j))_{j=1}^{k}$, see Figure \ref{fig:mu_random_permutations}. (A random permutation sampled this way from $\mu$ is called a \emph{$\mu$-random permutation}.) 

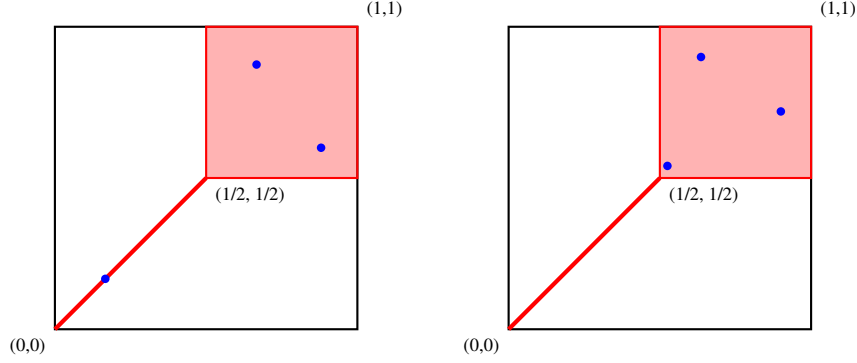
\begin{figure}\centering
\begin{tikzpicture}[scale=4]
    
    \begin{scope}
        \draw[thin, gray!20] (0,0) grid (1,1);

        \draw[thick] (0,0) -- (1,0) -- (1,1) -- (0,1) -- cycle;

        \fill[red, opacity=0.3] (0.5, 0.5) rectangle (1, 1);
        \draw[red, thick] (0.5, 0.5) rectangle (1, 1);

        \draw[red, ultra thick] (0,0) -- (0.5, 0.5);

        \fill[blue] (1/6, 1/6) circle (0.4pt);
        \fill[blue] (2/3, 0.875) circle (0.4pt);
        \fill[blue] (0.88, 0.6) circle (0.4pt);

        \node[below left, font=\scriptsize] at (0,0) {(0,0)};
        \node[above right, font=\scriptsize] at (1,1) {(1,1)};
        \node[below right, font=\scriptsize] at (0.5, 0.5) {(1/2, 1/2)};
    \end{scope}

    \begin{scope}[xshift=1.5cm] 
        \draw[thin, gray!20] (0,0) grid (1,1);

        \draw[thick] (0,0) -- (1,0) -- (1,1) -- (0,1) -- cycle;

        \fill[red, opacity=0.3] (0.5, 0.5) rectangle (1, 1);
        \draw[red, thick] (0.5, 0.5) rectangle (1, 1);

        \draw[red, ultra thick] (0,0) -- (0.5, 0.5);

        \fill[blue] (0.525, 0.54) circle (0.4pt);
        \fill[blue] (0.636, 0.9) circle (0.4pt);
        \fill[blue] (0.9, 0.72) circle (0.4pt);

        \node[below left, font=\scriptsize] at (0,0) {(0,0)};
        \node[above right, font=\scriptsize] at (1,1) {(1,1)};
        \node[below right, font=\scriptsize] at (0.5, 0.5) {(1/2, 1/2)};
    \end{scope}
    
\end{tikzpicture}
\caption{A permuton $\mu$ with two different 3-samples from it. With respect to the Lebesgue measure, $\mu$ has nonzero singular and absolutely continuous parts: mass 1/2 is placed uniformly on the segment ((0, 0), (1/2, 1/2)) and mass 1/2 uniformly on the subsquare shaded in red. The two 3-samples visualize the two essentially different ways a $\mu$-random permutation of length 3 might equal 132: either one point lands in the singular part and the remaining two determine the pattern 21 in the absolutely continuous part, or all of them lands in the absolutely continuous part and determine 132 in there. Writing up the probabilities for both of these outcomes, it is straightforward to check $t(132, \mu)=3/16+1/48=5/24$.} \label{fig:mu_random_permutations}
\end{figure}

A sequence of permutations $(\pi_n)_{n=1}^{\infty}$ is convergent if $|\pi_n|\to\infty$ and $t(\sigma, \pi_n)$ is convergent for each pattern $\sigma$. A fundamental result in permuton theory \cite{HOPPEN201393} is that a convergent permutation sequence determines a unique permuton whose pattern densities coincide with the corresponding limit densities. Conversely, we also know that any permuton $\mu$ is the limit of a permutation sequence, conveniently proved by using $\mu$-random permutations and showing that they are very close to $\mu$ with probability tending rapidly to one. This closeness is appropriately measured by the \emph{rectangular distance},
defined by
$$\ds(\mu, \nu)=\sup_{R=[a,b]\times [c, d]}|\mu(R)-\nu(R)|,$$
a metric which metrizes the topology induced by convergence in pattern densities and also the topology of weak convergence restricted to permutons. (For permutons, the supremum is actually a maximum.) By \cite[Lemma~4.2]{HOPPEN201393} a $\mu$-random permutation of length $k$ (or more precisely, the naturally associated permuton, properly defined in Section \ref{sec:main_results}) is at most $k^{-1/4}$ apart from $\mu$ in the rectangular distance with high probability. The proof given there actually guarantees distance $O((\log k/k)^{1/2})$ asymptotically almost surely, already shown in \cite[Lemma~4.9]{cooper2002quasirandompermutations} without using the language of permutons.

In this paper, we study the following question originally asked by Bálint Virág and communicated to the author by Balázs Ráth: to what extent can the approximation errors coming from random permutations be improved by a more careful choice? In other words, how well a given permuton can be approximated in $\ds$ by finite permutations? This can be thought of as the permuton version of the central question of measure-theoretic discrepancy theory.

\subsection{The main results} \label{sec:main_results}

The permuton $\widehat{\pi}$ naturally associated with $\pi\in\Sym(n)$ is defined as the uniform measure on $\bigcup_{i=1}^{n} \left[\frac{i-1}{n}, \frac{i}{n}\right]\times \left[\frac{\pi(i)-1}{n}, \frac{\pi(i)}{n}\right]$, see Figure \ref{fig:approx_example}. (We refer to these $n$ grid squares as the \emph{base squares} of either $\pi$ or $\widehat{\pi}$.) We put 
$$D_n(\mu)=\min_{\pi \in \Sym(n)} \ds(\mu, \widehat{\pi}).$$
\begin{figure}\centering
\begin{tikzpicture}[scale=4]
    \fill[red, opacity=0.3] (0.5, 0.5) rectangle (1, 1);
    
    \draw[red, ultra thick] (0,0) -- (0.5, 0.5);

    \begin{scope}[fill=blue!40, opacity=0.5]
        \foreach \col/\row in {0/0, 1/1, 2/2, 3/3, 4/7, 5/6, 6/5, 7/4} {
            \fill (\col*0.125, \row*0.125) rectangle ++(0.125, 0.125);
        }
    \end{scope}

    \draw[thick] (0,0) rectangle (1,1);

    \draw[red, thick] (0.5, 0.5) rectangle (1, 1);
    \draw[green, thick] (0.5,0.5) rectangle (0.75,0.75);

\end{tikzpicture}
\caption{The permuton $\mu$ from the previous figure approximated by the permuton $\widehat{\pi}$ (shaded blue) for $\pi=12348765$. The rectangle witnessing $\ds(\mu, \widehat{\pi})=1/8$ is shown with green boundary. The best approximation is given by $\pi'=12345768$, showing $D_8(\mu)=3/32$.} \label{fig:approx_example}
\end{figure}
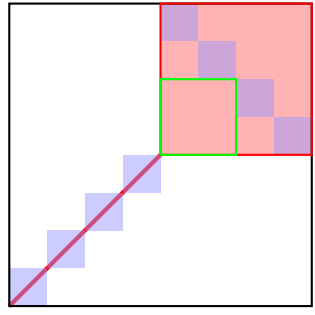

Established results in measure-theoretic discrepancy theory \cite{Chen1985, Aistleitner2018} (to be explicitly stated in Section 2), partially coming from bounds given to Tusnády's problem on combinatorial and geometric discrepancy \cite{Nikolov2017} can be tailored to imply some bounds on $D_n$. Given the aforementioned results, the proofs are very simple: as we will see, approximation with finitely supported measures (as in discrepancy theory) has the same error decay rate as $D_n(\mu)$. We state the following theorem among the main results nevertheless for future reference.

\begin{theorem} \label{thm:universal_approximability}
    For any permuton $\mu$,
    $$D_n(\mu)=O(\log^{3/2} n/n).$$
    If $\mu$ is absolutely continuous,
    $$D_n(\mu)=\Omega(\log^{1/2} n/n).$$
    If $\supp \mu$ can be covered by finitely many curves that are monotone in both coordinates, then
    $$D_n(\mu)=\Theta(1/n).$$
\end{theorem}

It is worth pointing out that the upper bound in Theorem~\ref{thm:universal_approximability} is significantly better than what is guaranteed by the standard approximation via $\mu$-random permutations discussed in the introduction.

The last statement of Theorem \ref{thm:universal_approximability} shows that singular permutons allow better rates of approximation, though as a basic triangle inequality based argument shows error $D_n(\mu)=o(1/n)$ can hold along a subsequence at best. (This is sharpened below in Theorem \ref{thm:regular_not_approximable}.) Such superlinear approximation even along a subsequence is only possible for permutons carried by measure-preserving bijections:

\begin{theorem}\label{thm:only_bijections_are_well_approximable}
    Assume that $D_n(\mu)=o(1/n)$ along a subsequence for a permuton $\mu$. Then there exists a measure-preserving bijection $f:[0,1]\to[0, 1]$ defined almost everywhere such that $\mu=\mu_f$, i.e. $\mu$ is the uniform measure on the graph of $f$.
\end{theorem}

The following theorem further refines the boundary of superlinear approximability and relates it to local regularity properties of $f$.

\begin{theorem} \label{thm:regular_not_approximable}
    Assume $\mu=\mu_f$ for an almost everywhere defined measure-preserving bijection.
    \begin{enumerate}
        \item If $D_n(\mu)=o(1/n)$ along a subsequence then $D_n(\mu)=\omega(1/n)$ along another subsequence.
        \item If $f$ is Hölder-$\alpha$ at a point $x$ for some $0<\alpha\leq 1$, then $D_n(\mu)=\Omega(n^{-1/\alpha}).$ In particular, if $f$ is Lipschitz at a point $x$, then no superlinear approximation is possible.
    \item The previous statement is sharp up to logarithmic factor: for every $0<\alpha\leq 1$ there exists an almost everywhere defined measure-preserving bijection $f$ which is Hölder-$\alpha$ at a dense set of points with Hausdorff dimension 1 and satisfies $D_n(\mu_f)=O\left(\frac{\log n} {n^{1/\alpha}}\right)$ along a subsequence.
    \end{enumerate}
\end{theorem}

We remark that for $\alpha=1$, the third statement is trivial even without the logarithmic factor, as demonstrated by $f(x)=x$. We further note that $D_n$ can decay arbitrarily fast using similar constructions to what is given in the proof of the third statement of Theorem \ref{thm:regular_not_approximable}, but of course we need to sacrifice regularity.

The Brownian separable permuton was introduced by \cite{Bassino2018} as the distributional limit of the uniform random separable permutation of length $n$. This concept was generalized in \cite{Bassino2017}. There, the biased Brownian separable permuton was introduced as a one-parameter family \((\mu^p)\), and a universality result was proved: under mild conditions, uniformly random permutations from substitution-closed permutation classes converge to some \(\mu^p\). This family of permutons attracted significant attention. As proved in \cite{Maazoun_2020}, $\mu^p$ is almost surely supported by the graph of a measure-preserving function, and has self-similarity properties, implying a fractal structure. Capitalizing on this result, we prove a result about its pointwise regularity:

\begin{lemma}\label{lemma:brownian_lip}
    Almost surely, the measure-preserving function whose graph supports the biased Brownian separable permuton $\mu^p$ is Lipschitz at some points.
\end{lemma}

Lemma \ref{lemma:brownian_lip}, coupled with Theorem \ref{thm:regular_not_approximable} directly yields the following theorem on the approximability of the biased Brownian separable permuton:

\begin{theorem}\label{thm:brownian_approx}
    For the biased Brownian separable permuton $\mu^p$ almost surely $D_n(\mu^p)=\Omega(1/n)$.
\end{theorem}

Note that the previous theorem does not tell the actual decay rate of $D_n(\mu^p)$, we highlight this as an open question:

\begin{question} \label{quest:brownian_decay}
    For the biased Brownian separable permuton $\mu^p$, what is the almost sure decay rate of $D_n(\mu^p)$? 
\end{question}

\begin{remark}
We note that the construction proving the third statement of Theorem \ref{thm:regular_not_approximable} and Lemma \ref{lemma:brownian_lip} both rely primarily on the self-similar structure of the corresponding permutons. Consequently, the statements of Lemma \ref{lemma:brownian_lip} and hence Theorem \ref{thm:brownian_approx} are also valid for some other statistically self-similar permutons, such as the recursive separable permuton introduced in \cite{RecursiveSeparablePermuton} as the permuton limit of a Markovian model of separable permutations, or the permutons associated with random automorphisms of the infinite rooted $d$-ary tree \cite{maga2025samplingentropypermutons}. We omit the proofs as they are essentially identical to the proof of Lemma \ref{lemma:brownian_lip}. We take note of the fact that while for the latter of these two random permutons, Question \ref{quest:brownian_decay} also seems difficult, for the former, the bound \(D_n(\mu)=O(1/n)\) along a subsequence follows straightforwardly from the compatibility of the self-similar structure with a regular grid.
\end{remark}

\subsection{Discussion of the choice of $\widehat{\pi}$}

We note that while it is customary to define $\widehat{\pi}$ as we defined it, it is not canonical and the essence of permuton theory does not depend on this particular choice. Notably, instead of putting the uniform measure to each of the $1/n$ grid squares corresponding to the permutation matrix of $\pi$, one could put there any appropriately rescaled permuton, not necessarily the same one for different squares, to get a large family of permutons naturally associated with $\pi$. Denoting this family by $\Perm(\pi)$, one finds that for any choice of $\widehat{\pi}_n'\in \Perm(\pi_n)$, the convergence $\pi_n\to \mu$ in terms of pattern densities is equivalent to $\widehat{\pi_n}'\to \mu$ in $\ds$. Thus passing to such an admissible alternative of $\widehat{\pi}$ indeed makes no difference for the big picture. However, at least formally it certainly affects the study of superlinear approximations. In the most relaxed form one could define 
$$D_n'(\mu)=\min_{\pi \in \Sym(n)} \inf_{\widehat{\pi}'\in \Perm(\pi)\text{ }} \ds(\mu, \widehat{\pi}').$$
and study approximability in terms of $D_n'$. Inspecting the proof, it is not difficult to see that Theorem \ref{thm:only_bijections_are_well_approximable} holds for $D_n'$ as well, however, the statements of Theorem \ref{thm:regular_not_approximable} are no longer valid, displayed for instance by $f(x)=x$, for which $D_n'(\mu_f)=0$ for every $n$. (It is simple to show that $D_n'$ vanishes for every $n$ only for the monotone permutons.) Thus $D_n$ and $D_n'$ indeed behave differently to some extent, we choose to study $D_n$ for aesthetic reasons, such as the duality-flavored statements of Theorem \ref{thm:regular_not_approximable}.

\subsection{Organization of the paper}
In Section \ref{sec:prelim} we introduce relevant notation and recall some results of discrepancy theory. In Section \ref{sec:univ_approx} we discuss their consequences concerning our problem and prove Theorem \ref{thm:universal_approximability}. Section \ref{sec:superlin_approx} is devoted to the proof of Theorem \ref{thm:only_bijections_are_well_approximable}. In Section \ref{sec:regular_approx}, Theorem \ref{thm:regular_not_approximable} is proved after some preparations. Section \ref{sec:brownian_separable} contains the proof of Lemma \ref{lemma:brownian_lip}.

\section{Preliminaries and notation} \label{sec:prelim}

The set of probability measures on a set $S$ is denoted by $\mathcal{P}(S)$. Then $\Perm\subset \mathcal{P}([0, 1]^2)$. We denote projections to the $x$-axis and $y$-axis by $\Pr_x$ and $\Pr_y$, respectively. For a measure-preserving function $f$, the permuton \(\mu_f\) is the pushforward of the Lebesgue measure under \(x\mapsto(x,f(x))\).

The central question of measure-theoretic discrepancy theory is that how well finitely supported uniform distributions can approximate a given measure. To quantify this, star discrepancy is usually defined between finite point sets and measures. Slightly abusing this language, we define the star discrepancy of measures $\mu, \nu \in \mathcal{P}([0, 1]^2)$ as
$${\ds}^*(\mu, \nu)=\sup_{R=[0, x]\times [0, y]}|\mu(R)-\nu(R)|.$$

The following lemma is immediate from the inclusion-exclusion principle:

\begin{lemma}
    For any $\mu, \nu \in \mathcal{P}([0, 1]^2)$
    $${\ds}^* (\mu, \nu) \leq {\ds}(\mu, \nu)\leq 4{\ds}^*(\mu, \nu).$$
\end{lemma}

Consequently, the approximability rates in $\ds$ and $\ds^*$ coincide up to constant factor.

For a finite point multiset $P$ we write
$$\mu_P = \frac{1}{|P|}\sum_{p\in P}\delta_p.$$

To conclude the preliminaries, we recall a few results of discrepancy theory which are relevant in our discussion. First, for the approximation of the Lebesgue measure, we have the following matching bounds on the rate of approximation:

\begin{theorem}[Two-dimensional case of {\cite{Halton1960}}]\label{thm:halton}
    For any  $n$ there exist $x_1, \dots, x_n\in [0, 1]^2$ with
    $${\ds}^*\left(\frac{1}{n}\sum_{i=1}^{n}\delta_{x_i}, \lambda\right)= O\left(\frac{\log n}{n}\right).$$
\end{theorem}

\begin{theorem}[{\cite[Theorem~2]{Schmidt_1972}}]\label{thm:schmidt}
    There exists $c>0$ such that for any $x_1, \dots, x_n\in [0, 1]^2$ we have
    $${\ds}^*\left(\frac{1}{n}\sum_{i=1}^{n}\delta_{x_i}, \lambda\right)\geq \frac{c\log n}{n}.$$
\end{theorem}

It should be noted that while \ref{thm:halton} extends to higher dimensions $k$ with $O(\log^{k-1}{n}/n)$, whether this is optimal is known as the Great Open Problem of discrepancy theory. On a related note, even in two dimensions it is unknown whether other measures can exhibit worse approximability properties, with the best known result being the following:

\begin{theorem}[{Two-dimensional case of \cite[Theorem~1]{Aistleitner2018}}]\label{thm:aistleitner}
    For any $\mu \in \mathcal{P}([0, 1]^2)$ and $n$ there exists $x_1, \dots, x_n\in[0, 1]^2$ with
    $${\ds}^*\left(\frac{1}{n}\sum_{i=1}^{n}\delta_{x_i}, \mu\right)= O\left(\frac{\log^{3/2}n}{n}\right).$$
\end{theorem}

The following statement appears in substantially greater generality in \cite{Chen1985}, allowing non-uniform finite distributions and concerning various norms of the discrepancy.

\begin{theorem}[{Two-dimensional case of \cite[Corollary~1]{Chen1985}}] \label{thm:chen}
    For any $\mu \in \mathcal{P}([0, 1]^2)$ absolutely continuous, there exists some $c(\mu)>0$ such that for any $x_1, \dots, x_n\in[0, 1]^2$
    $${\ds}^*\left(\frac{1}{n}\sum_{i=1}^{n}\delta_{x_i}, \mu\right)> \frac{c(\mu)\log^{1/2} n}{n}.$$
\end{theorem}

\section{Universal approximation} \label{sec:univ_approx}

The rates given by Theorems \ref{thm:aistleitner}--\ref{thm:chen} are precisely in line with the statement of Theorem \ref{thm:universal_approximability}. To convert them to statements about $D_n$, we just need to convert point sets into permutons associated with permutations with small error in $\ds$. To this end, abusing the notation, we identify a permutation $\pi$ with the point set $\{(i/n, \pi(i)/n): i\in [n]\}\subseteq [0, 1]^2$. The following statement is immediate from the definitions:

\begin{lemma} \label{lemma:step_permuton_permutation_point_measure_equivalence}
    For any $\pi\in \Sym(n)$
    $$1/n\leq \ds(\widehat\pi, \mu_{\pi})\leq 4/n.$$
\end{lemma}

A finite point set in $[0, 1]^2$ is \emph{generic} if its points have pairwise distinct $x$ and pairwise distinct $y$ coordinates. For a generic point set $P$ in $[0, 1]^2$, the pattern $\pi(P)$ it determines is straightforward to define: we move from left to right and read the height ordering to get a permutation, see Figure \ref{fig:permutation_regularization}. For a non-generic point set, or more generally for a multiset of points, ties may occur both in the left-to-right ordering and in the height ordering; we resolve all such ties arbitrarily.

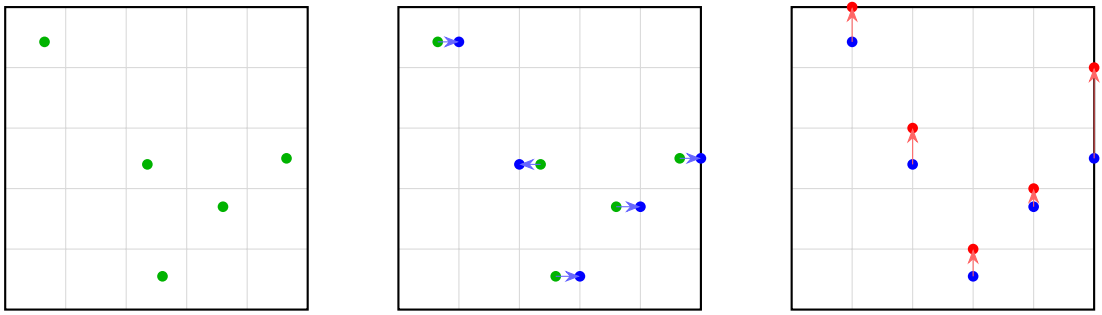
\begin{figure}[htb]\centering
\begin{tikzpicture}[scale=4]

    \tikzset{
        panelgrid/.style={thin, gray!30, step=0.2},
        arrowstyle/.style={-{Stealth[scale=1.2]}, thin}
    }

    \begin{scope}[xshift=0cm]
        \draw[panelgrid] (0,0) grid (1,1);
        \draw[thick] (0,0) rectangle (1,1);
        
        \foreach \p in {(0.13, 0.885), (0.47, 0.48), (0.52, 0.11), (0.72, 0.34), (0.93, 0.5)} {
            \fill[green!70!black] \p circle (0.5pt);
        }
    \end{scope}

    \begin{scope}[xshift=1.3cm]
        \draw[panelgrid] (0,0) grid (1,1);
        \draw[thick] (0,0) rectangle (1,1);

        \coordinate (A1) at (0.13, 0.885); \coordinate (B1) at (0.2, 0.885);
        \coordinate (A2) at (0.47, 0.48);  \coordinate (B2) at (0.4, 0.48);
        \coordinate (A3) at (0.52, 0.11);  \coordinate (B3) at (0.6, 0.11);
        \coordinate (A4) at (0.72, 0.34);  \coordinate (B4) at (0.8, 0.34);
        \coordinate (A5) at (0.93, 0.5);   \coordinate (B5) at (1.0, 0.5);

        \foreach \i in {1,2,3,4,5} {
            \fill[green!70!black] (A\i) circle (0.5pt);
            \fill[blue] (B\i) circle (0.5pt);
            \draw[arrowstyle, blue!60] (A\i) -- (B\i);
        }
    \end{scope}

    \begin{scope}[xshift=2.6cm]
        \draw[panelgrid] (0,0) grid (1,1);
        \draw[thick] (0,0) rectangle (1,1);

        \coordinate (B1) at (0.2, 0.885); \coordinate (C1) at (0.2, 1.0);
        \coordinate (B2) at (0.4, 0.48);  \coordinate (C2) at (0.4, 0.6);
        \coordinate (B3) at (0.6, 0.11);  \coordinate (C3) at (0.6, 0.2);
        \coordinate (B4) at (0.8, 0.34);  \coordinate (C4) at (0.8, 0.4);
        \coordinate (B5) at (1.0, 0.5);   \coordinate (C5) at (1.0, 0.8);

        \foreach \i in {1,2,3,4,5} {
            \fill[blue] (B\i) circle (0.5pt);
            \fill[red] (C\i) circle (0.5pt);
            \draw[arrowstyle, red!60] (B\i) -- (C\i);
        }
    \end{scope}

\end{tikzpicture}
\caption{The first panel shows a generic point set $P$ with $\pi(P)=53124$. The second and third panel shows how $\pi(P)$ (as a point set) is obtained in the proof of Lemma \ref{lemma:ds_permutation_regularization} from $P$ in two steps: first we regularize the $x$ coordinates to obtain $P_{\rightarrow}$, then the $y$ coordinates to obtain $\pi(P)$.} \label{fig:permutation_regularization}
\end{figure}

In general $\mu_P$ and $\mu_{\pi(P)}$ can be far from each other. However:

\begin{lemma} \label{lemma:ds_permutation_regularization}
    Put $n=|P|$. If ${\ds}^*(\mu_P, \mu)<C/n$ for a permuton $\mu$ and some $C>0$, then
    $${\ds}^*(\mu_P, \mu_{\pi(P)})<\frac{2C+3}{n}$$
\end{lemma}

\begin{proof}
    Note that $\pi(P)$ as a point set can be obtained by first horizontally displacing points without changing their relative order by $x$-coordinate, then vertically displacing them without changing their relative order by $y$-coordinate. Denote the intermediate set by $P_{\rightarrow}$ (second panel of Figure \ref{fig:permutation_regularization}).
    
    Consider $R=[0, x]\times [0, 1]$. Putting $n_S(R)$ for the number of points of $S$ in $R$, we have
    $$|n_P(R)-nx|< C,$$
    while
    $$|n_{P_{\rightarrow}}(R)-nx|\leq 1,$$
    implying
    \begin{equation}\label{eq:point_difference}
    |n_P(R)- n_{P_{\rightarrow}}(R)|<C+1.
    \end{equation}
    As $P_{\rightarrow}$ is obtained from $P$ by moving points only along the $x$-axis in an order-preserving manner, the difference between the quantities $n_P(R), n_{P_{\rightarrow}}(R)$ solely comes from points entering $R$ or points exiting $R$, no points can cross the boundary in opposite directions simultaneously. This implies that
    \eqref{eq:point_difference} is valid for any $R'=[0, x]\times [0, y]$ in place of $R$, as there is no flux between $R', R\setminus R'$ either due to the horizontal displacements. Consequently,
    $${\ds}^*(\mu_P, \mu_{P_{\rightarrow}})<\frac{C+1}{n}.$$
    We argue the same way for replacing $P_{\rightarrow}$ by $\pi(P)$.    
\end{proof}

\begin{proof}[Proof of Theorem \ref{thm:universal_approximability}]
    For the first statement, consider the finite point set guaranteed by Theorem \ref{thm:aistleitner}. Using Lemma \ref{lemma:ds_permutation_regularization}, we can pass to a nearby permutation point set $\pi$ without changing the error rate, which in turn is in the $O(1/n)$ proximity of $\widehat{\pi}$ by Lemma \ref{lemma:step_permuton_permutation_point_measure_equivalence}. 

    The second statement follows immediately from Theorem \ref{thm:chen} and Lemma \ref{lemma:step_permuton_permutation_point_measure_equivalence}.

    For the third statement cover $\supp \mu$ with injective, coordinatewise monotone curves $(\gamma_i)_{i=1}^{k}$, parametrized by the closed intervals $(I_i)_{i=1}^k$. We can assume that these are ordered, $0<I_1<\dots <I_k$. We want these pieces to have disjoint ranges, so we put
    $$A_i = I_i \setminus \gamma_i^{-1}\left(\bigcup_{j=1}^{i-1}\gamma_j(I_j))\right).$$
    Now $\gamma:\bigcup_{i=1}^{k} A_i\to[0, 1]^2$ defined with $\gamma(x)=\gamma_i(x)$ once $x\in A_i$ is an injective function which is coordinatewise monotone restricted to each $A_i$, and $\gamma(\bigcup_{i=1}^{k}A_i)=\bigcup_{i=1}^{k}\gamma_i(I_i)\supseteq \supp \mu$. Below $\gamma(I)$ for an interval $I$ is understood as $\gamma(I\cap \bigcup_{i=1}^{k}A_i)$.

    Fix now $n$. 
    We claim that for $j=1, 2, \dots, n$ there exists $s_j \in \bigcup_{i=1}^{k}A_i$ with
    $$\mu(\gamma([0, s_j]))\in \left(\frac{j}{n}-\frac{1}{n^2}, \frac{j}{n}+\frac{1}{n^2}\right).$$
    Indeed, otherwise by interval-halving we can find a nested sequence of closed intervals $J_m$ with length converging to 0 such that $\mu(\gamma(J_m))>\frac{2}{n^2}$ for each $m$. Then for the common point $x$ of these intervals $\mu(\gamma(x))\geq \frac{2}{n^2}$, which is impossible because uniform marginals imply that $\mu$ has no atoms.

    Put now $P=\{\gamma(s_j)\}_{j=1}^n$. Note that by monotonicity, for any rectangle $R$ and $i=1, 2, \dots, k$ there exists some interval $I_{i, R}$ such that $R\cap \gamma_i(A_i)=\gamma(I_{i, R})$ and hence $$\mu(R)=\mu\left(\gamma\left(\bigcup_{i=1}^k I_{i, R}\right)\right)=\sum_{i=1}^{k}\mu(\gamma(I_{i, R})), \quad\mu_P(R)=\mu_P\left(\gamma\left(\bigcup_{i=1}^k I_{i, R}\right)\right)=\sum_{i=1}^{k}\mu_P(\gamma(I_{i, R})).$$

    Then by triangle inequality
    $$\ds (\mu, \mu_P)\leq \sup_R \left(\sum_{i=1}^{k}|\mu(\gamma(I_{i, R}))-\mu_P(\gamma(I_{i, R}))|\right).$$
    However, by the choice of $P$, for any $R, i$ 
    $$|\mu(\gamma(I_{i, R}))-\mu_P(\gamma(I_{i, R}))|\leq 2/n+2/n^2.$$
    This implies $\ds(\mu, \mu_P)\leq 2k/n+2k/n^2$ and then one can use Lemma \ref{lemma:ds_permutation_regularization} to replace $P$ by $\pi(P)$ and \ref{lemma:step_permuton_permutation_point_measure_equivalence} to replace $\pi(P)$ by $\widehat{\pi(P)}$ to ultimately obtain $D_n(\mu)=O(1/n)$. The lower bound on approximability follows from the Statement 1 of Theorem \ref{thm:regular_not_approximable} proved later: $O(1/n)$ approximability along the full sequence rules out $o(1/n)$ along a subsequence.
\end{proof}

\section{Superlinear approximation along a subsequence} \label{sec:superlin_approx}

\begin{proof}[Proof of Theorem \ref{thm:only_bijections_are_well_approximable}]
    Fix a subsequence $n_1, n_2, \dots$ in accordance with the assumption of the theorem, that is $D_{n_k}(\mu)<\alpha_{k}/n_k$ with $\alpha_{k}\to 0$.
    By discarding certain entries of the sequence, we may assume that $\alpha_{k}$ is decreasing and $\sum_{k=1}^{\infty} \alpha_{k} < \infty$. We may also assume that $\frac{n_{k+1}}{n_k}>2^k$. To ease notation, set $\mu_{k} = \widehat{\pi_{k}}$ with $\pi_{k}\in \Sym(n_k)$ realizing $D_{n_k}$.

     Recall that a base square of $\pi_{k}$ is a grid square of size $n_k^{-1}\times n_k^{-1}$ contained by the support of $\widehat{\pi_k}$. For $x\in [0, 1]\setminus \mathbb{Q}$, let $Q_k(x)=I_k(x)\times J_k(x)$ be the unique base square of $\pi_{k}$ for which $x\in I_k(x)$. (For rational $x$, this would not be unique.) We say that the sequence $Q_k(x)$ is \emph{eventually nested} if $Q_{k+1}(x)\subseteq Q_k(x)$ for sufficiently large $k$.

    \begin{claim}\label{claim:stab}
        For almost every $x\in [0, 1]$, $Q_k(x)$ is eventually nested.
    \end{claim}

    \begin{proof}[Proof of Claim \ref{claim:stab}]
        Consider a base square $Q_k = I_k \times J_k$ of $\pi_k$. Define the rectangles $R_k^1, R_k^2$ to be non-overlapping with $Q_k$ such that $$I_k\times [0, 1] = R_k^1 \cup Q_k\cup R_k^2.$$
        Plugging in these rectangles into the definition of $d_\square$ verifies that for $i=1, 2$
        \begin{equation}\label{eq:claim_stab}
        \mu(R_k^i)\leq \frac{\alpha_k}{n_k}.
        \end{equation}
        Now consider base squares of $\pi_{k+1}$ over $I_k$, the $x$-projections of these are $1/n_{k+1}$ grid intervals. As $n_k$ may not be a divisor of $n_{k+1}$, these grid intervals may not be compatible with the $1/n_k$ grid interval $I_k$, however, $I_k$ fully contains at least $\frac{n_{k+1}}{n_k}-2$ of them. Among the corresponding base squares there can be at most 2 which intersects $Q_k$ nontrivially, all others are either contained in $R_k^1$, $R_k^2$, or $Q_k$. Observe that $R_k^i$ ($i=1, 2$) can contain at most $2\alpha_k \frac{n_{k+1}}{n_k}$ of them. Indeed, otherwise in $R_k^i$, the difference of $\mu$ and $\mu_{k+1}$ would be at least $\frac{\alpha_k}{n_k}$ due to \eqref{eq:claim_stab}, contradicting $d_\square(\mu_{k+1}, \mu)\leq \frac{\alpha_{k+1}}{n_{k+1}}$. To sum up, we conclude the following for $Q_{k+1}(x)$ if $x\in I_k$ (and hence $Q_k = Q_k(x)$:
        \begin{itemize}
            \item on a set of measure at most $\frac{2}{n_{k+1}}$, we have no control over $Q_{k+1}(x)$, this measure corresponds to the overhanging parts of the $1/n_{k+1}$ grid intervals.
            \item on a set of measure at most $\frac{2}{n_{k+1}}$, $Q_{k+1}(x)$ partially intersects $Q_k(x)$, hence it is not contained in it.
            \item on a set of measure at most $\frac{4\alpha_k}{n_k}$, $Q_{k+1}$ is contained in either $R_k^1$ or $R_k^2$, hence it is disjoint from $Q_k$.
            \item in the remainder of $I_k$, $Q_{k+1}(x)\subseteq Q_k(x)$
        \end{itemize}
        Hence in $I_k$, $Q_{k+1}(x)\subseteq Q_k(x)$ apart from measure at most $\frac{4}{n_{k+1}} + \frac{4\alpha_k}{n_k}$.
        Multiplying by $n_k$ to take into account all possible choices of $I_k$, $Q_{k+1}(x)\subseteq Q_k(x)$ fails in $[0,1]$ at a set of measure at most
        $$\frac{4n_k}{n_{k+1}} + 4\alpha_k.$$
        By the assumptions on $(\alpha_k)$ and $(n_k)$, the resulting series is summable, thus by the Borel--Cantelli lemma, we have that $Q_{k+1}(x)\subseteq Q_k(x)$ holds for almost every $x$ for sufficiently large $k$. This proves the claim.
    \end{proof}

    For almost every irrational $x$, whose existence is guaranteed by the claim, denote by $k(x)$ the minimal index from where $Q_k(x)$ is nested. Since the intervals \(J_j(x)\) are nested and have lengths \(1/n_j\to0\), their intersection consists of a single point; define \(f(x)\) to be this point.
     
    Observe that the function \(f\) is injective outside \(f^{-1}(\mathbb Q)\). Indeed, for $x\neq y$ and sufficiently large $K>k(x), k(y)$, we have $J_K(x)\neq J_K(y)$, thus $f(x)=f(y)$ might happen only if these values coincide with the shared endpoint of $J_K(x)$ and $J_K(y)$.
    
    After defining \(f\) arbitrarily on the remaining null set, let \(\mu_f\) be the pushforward of the Lebesgue measure under \(x\mapsto(x,f(x))\). We claim that it is the weak limit of the sequence $\mu_{k}$. This would conclude the proof: as the limit is unique, we obtain $\mu_f = \mu$, which directly implies that $f$ is a measure-preserving function, and by omitting the null-set $f^{-1}(\mathbb{Q})$ from its domain, we get a measure-preserving bijection due to our observation about $f$ being almost injective.

    By the Portmanteau theorem, it suffices to show that for any open set $G\subseteq [0, 1]^2$ we have
    \begin{equation}\label{eq:portmanteau}
    \liminf_{k\to\infty} \mu_k(G)\geq \mu_f(G).
    \end{equation}
    Fix $G$, and let $H=\{x:\ (x, f(x))\in G\}$. If $\lambda(H)=0$, then $\mu_f(G)=0$ and we have nothing to prove. Otherwise for any $\varepsilon>0$ for every large enough integer $K$ we can find $H'\subseteq H$ with the properties 
    \begin{itemize}
    \item $\lambda(H')>\lambda(H)-\varepsilon$,
    \item for any $x\in H'$, the square with sides $\frac{2}{n_K}$ centered at $(x, f(x))$ is contained in $G$,
    \item for any $x\in H'$, $k(x)\leq K$.
    \end{itemize}
    (By taking $K\to \infty$, both the second and third properties hold for each $x\in H$, which can be replaced by a finite $K$ with arbitrarily small loss of measure.)

    Now it is clear that $Q_K(x)\subseteq G$ for $x\in H'$ and $H'\subseteq \bigcup_{x\in H'} I_K(x)$. This yields that there are at least $\lambda(H')n_K$ distinct intervals in this union, or equivalently, at least $\lambda(H')n_K$ distinct base squares of $\pi_K$ in $G$. As $\mu_K$ puts measure $1/n_K$ to each such square, $\mu_K(G)\geq \lambda(H)-\varepsilon$. This proves \eqref{eq:portmanteau}, hence concludes the proof of the theorem.
    
\end{proof}

\section{Approximation under regularity conditions} \label{sec:regular_approx}

Our goal in this section is to prove Theorem \ref{thm:regular_not_approximable}.
We will work with certain fractal constructions (Figure \ref{fig:fractal_scheme}). Given a sequence of integers $n_1< n_2< \dots $ and $N_k=\prod_{i=1}^{k}n_k$, the $k$th level grid intervals are the ones with length $N_k^{-1}$ in $[0, 1]$. Denote the family of such intervals by $\mathcal{I}_k$. (For the sake of completeness, $\mathcal{I}_0=\{[0, 1]\}$.) The subfamily containing those $I\in \mathcal{I}_k$ for which the unique $I'\in \mathcal{I}_{k-1}$ with $I'\supseteq I$ is such that $I\subseteq [\min I' + r|I'|, \max I'-r|I'|]$ is denoted by $\mathcal{I}_{k, r}$. In this case, we say that $I$ is a direct $r$-descendant of $I'$. Finally,
$$F_{k_0, r} = \bigcap_{k=k_0}^{\infty}\bigcup\mathcal{I}_{k, r}.$$
The primary role of $k_0$, i.e., not taking the intersection above over all \(k\geq1\) is to guarantee a nonempty intersection. 

\begin{figure}[htb]\centering
\begin{tikzpicture}[xscale=12, yscale=1.5]

    \draw[thick] (0, 0) -- (1, 0);
    \draw[thick] (0, 0.1) -- (0, -0.1);
    \draw[thick] (1, 0.1) -- (1, -0.1);

    \begin{scope}[yshift=-1cm]
        \draw[thick] (0,0) -- (1,0);
        \foreach \x in {0, 0.25, 0.5, 0.75, 1}
            \draw[thick] (\x, 0.1) -- (\x, -0.1);
        \draw[ultra thick, red] (0.25,0) -- (0.75,0);
    \end{scope}

    \begin{scope}[yshift=-2cm]
        \draw[thick] (0,0) -- (1,0);
        \foreach \x in {0,1,...,28}
            \draw[thin, gray!60] (\x/28, 0.04) -- (\x/28, -0.04);
        \foreach \x in {0, 0.25, 0.5, 0.75, 1}
            \draw[black, semithick] (\x, 0.08) -- (\x, -0.08);
        \draw[ultra thick, red] (9/28, 0) -- (12/28, 0);
        \draw[ultra thick, red] (16/28, 0) -- (19/28, 0);
    \end{scope}

    \begin{scope}[yshift=-3cm]
        \draw[thick] (0,0) -- (1,0);
        
        \foreach \x in {0,1,...,280}
            \draw[ultra thin, gray!40] (\x/280, 0.02) -- (\x/280, -0.02);
            
        \foreach \x in {0,1,...,28}
            \draw[thin, gray!80] (\x/28, 0.05) -- (\x/28, -0.05);

        \foreach \x in {0, 0.25, 0.5, 0.75, 1}
            \draw[black, semithick] (\x, 0.1) -- (\x, -0.1);

        
        \foreach \k in {9, 10, 11, 16, 17, 18} {
            \draw[ultra thick, red] (10*\k/280 + 3/280, 0) -- (10*\k/280 + 7/280, 0);
        }
        
    \end{scope}

\end{tikzpicture}
\caption{A few steps of the described fractal scheme for $n_1=4, n_2=7, n_3=10$. With red, we see the first three initial segments of the infinite intersection defining $F_{1, 1/4}$.} \label{fig:fractal_scheme}
\end{figure}
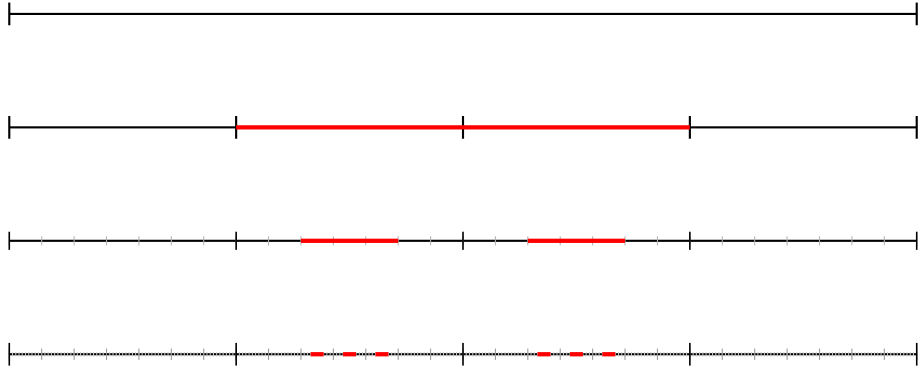

We obtain the following:

\begin{lemma} \label{lemma:dimension}
    If $k_0$ is large enough to guarantee that $F_{k_0,r}$ is nonempty, then for any $r<1/2$ we have $\dim_H F_{k_0, r}=1.$
\end{lemma}

\begin{proof}
    The proof is a highly standard one in geometric measure theory, yet we include it for the sake of completeness. We define a probability measure $\nu$ on $F_{k_0, r}$ inductively: $\nu(F_{k_0, r}\cap I)=N_k^{-1}$ for $k<k_0$ and $I\in \mathcal{I}_k$, and for $I\in \mathcal{I}_{k, r}$ with $k\geq k_0$, if $I$ is the direct $r$-descendant of $I'\in \mathcal{I}_{k-1}$ then
    $$\nu(F_{k_0, r}\cap I)= \frac{\nu(F_{k_0, r}\cap I')}{\text{no. of direct $r$-descendants of $I'$}}.$$
    With this definition, it is immediate that $\nu$ is consistently defined on sets of the form $F_{k_0, r}\cap I$, and thus it extends uniquely to all Borel subsets of $F_{k_0, r}$ as a measure by Carathéodory's extension theorem. It is also straightforward to check that we have for $k\geq k_0$ and $I\in \mathcal{I}_{k, r}$
    $$\nu(F_{k_0, r}\cap I)=\frac{1}{\text{size of $\mathcal{I}_{k, r}$}}\sim \frac{1}{(1-2r)^{k-k_0}N_k}.$$

    We would like to apply the mass distribution principle, to which end we fix an interval $J\subseteq [0, 1]$ with $|J|<1/N_{k_0}$. There exists unique $k>k_0$ such that
    $$\frac{l}{N_k}\leq|J|<\frac{l+1}{N_k}$$
    for some $l\in \{1, 2, \dots, n_k-1\}$. Then $J$ is contained in the union of at most $l+1$ adjacent intervals of $\mathcal{I}_k$, hence the same holds for $J\cap F_{k_0, r}$ as well. Consequently,
    $$\nu(J\cap F_{k_0, r})\leq \frac{l+1}{(1-2r)^{k-k_0}N_k}.$$
    Thus mass distribution principle guarantees $\dim_H F_{k_0, r}\geq d$ once there exists $C$ with 
    $$\frac{l+1}{(1-2r)^{k-k_0}N_k}\leq \frac{Cl^d}{N_k^d}.$$
    Writing $l+1\leq 2l$, and replacing $l$ by $n_k$ it suffices to demonstrate the existence of $C$ with
    \begin{equation}\label{eq:mdp}
    2(1-2r)^{-(k-k_0)}\leq C\left(\frac{N_k}{n_k}\right)^{1-d}=CN_{k-1}^{1-d}.
    \end{equation}
    Note that as $n_k$ is a strictly increasing sequence, $N_k$ grows superexponentially, i.e., for any $a>0$ we have $N_k>a^k$ once $k$ is large enough. Thus the same holds for $N_{k-1}^{1-d}$. Thus \eqref{eq:mdp} is satisfied for arbitrary $r<1/2$ and $d<1$: even if we omit the factor $C$, the right hand side will eventually exceed the left hand side, and the finitely many exceptions can be handled by choosing $C$ large enough.
\end{proof}

\begin{proof}[Proof of Theorem \ref{thm:regular_not_approximable}]
    For the first statement, assume that $(n_k)_{k=1}^{\infty}$ is a subsequence along which $D_{n_k}(\mu)=\alpha_k/n_k$ with $\alpha_k\to 0$, witnessed by permutations $\pi_k\in \Sym(n_k)$. Take a base square $Q$ of $\pi_k$ and consider some $m_k=ln_k$ for $l$ to be fixed later. Then for any $\pi\in \Sym(m_k)$, merely observing rectangles fully in $Q$, Theorem \ref{thm:schmidt} implies that
    $$\ds(\widehat{\pi}, \widehat{\pi_k})\geq \frac{C_1 \log l}{n_kl}.$$
    (This follows immediately if $Q$ contains precisely $l$ many base squares of $\pi$. Otherwise the error is even larger.)
    Consequently
    $$\ds(\widehat{\pi}, \mu)\geq \frac{C_1 \log l}{n_kl}-\frac{\alpha_k}{n_k}.$$
    With $l\sim \alpha_k^{-1}$,
    $$\ds(\widehat{\pi}, \mu)\geq \frac{C_1 \alpha_k\log \alpha_k^{-1}}{n_k}-\frac{\alpha_k}{n_k}=(n_k/\alpha_k)^{-1}(C_1\log \alpha_k^{-1}-1)=\omega((n_k l)^{-1})=\omega(1/m_k).$$
    Thus such a choice of $l$ and hence $m_k$ is good for our purposes.

    For the second statement, fix $\pi\in \Sym(n)$ and let $\nu = \widehat{\pi}$. Assume that $f$ is Hölder-$\alpha$ in $x$ with Hölder constant $C\geq 1$. Consider (one of) the base square(s) $Q$ of $\pi$ above $x$. Moreover, consider the two-sided Hölder cusp (Figure \ref{fig:holder_cusp})
\[
\{(t, y):|y-f(x)|\leq C|t-x|^\alpha\}.
\]

\begin{figure}[htb]\centering
\begin{tikzpicture}[scale=5]
    \def\px{0.5}
    \def\py{0.5}
    \def\len{1/9}
    \def\recty{0.5 + 1/3} 

    \begin{scope}
        \clip (0,0) rectangle (1,1);

        \fill[blue!15, domain=0:1, samples=300] 
            plot (\x, {\py + sqrt(abs(\x - \px))}) -- 
            plot [domain=1:0] (\x, {\py - sqrt(abs(\x - \px))}) -- cycle;

        \draw[blue!40, thin, domain=0:1, samples=300] 
            plot (\x, {\py + sqrt(abs(\x - \px))});
        \draw[blue!40, thin, domain=0:1, samples=300] 
            plot (\x, {\py - sqrt(abs(\x - \px))});
    \end{scope}

    \draw[dashed, gray!80] (\px, 0) -- (\px, 1);

    \draw[thick, black] (\px, \recty) rectangle ({\px + \len}, 1);
    
    \fill[red, opacity=0.5] (\px, \recty) rectangle ({\px + \len}, 1);

    \draw[thick] (0,0) rectangle (1,1);

    \fill[black] (\px, \py) circle (0.4pt);
    \node[right, font=\scriptsize] at (\px, \py) {$(x, f(x))$};

\end{tikzpicture}
\caption{The Hölder cusp (blue) for $C=1, \alpha=1/2$ and the optimal $R_t$ (red) constructed in the proof in the case when $(x, f(x))$ is the center of $Q$. It is simple to see that regardless of the location of $(x, f(x))$ in $Q$, a rectangle of this size can be found in the complement of the Hölder cusp. (The worst-case scenario being when $(x, f(x))$ is the midpoint of a vertical side of $Q$.)} \label{fig:holder_cusp}
\end{figure}
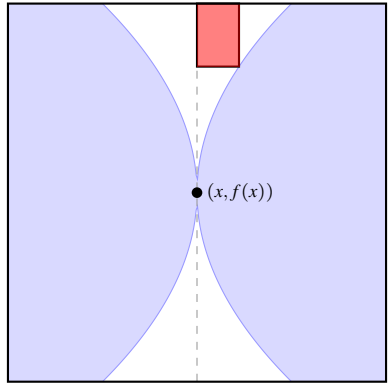

$Q$ contains at least 1 region which is disjoint from this cusp and hence from the support of $\mu$. The idea is that we try to put a rectangle as large as possible into one of these regions to witness that $d_\square(\mu, \nu)$ is large.

    From now on we only have to do straightforward calculation. The larger of the regions disjoint from the cusp can accommodate a rectangle $R_t$ with side lengths $t, \frac{1}{2n}-Ct^{\alpha}$, for any $0<t<\frac{1}{2n}$. By differentiation, we get that its area is maximal for
    $$t=\left(\frac{1}{2Cn(\alpha+1)}\right)^{\frac{1}{\alpha}}.$$
    (By $C\geq 1$, this value of $t$ is indeed in the interval $(0, 1/2n)$.) For this value, 
    $$\lambda(R_t)\geq \frac{c}{n^{1+1/\alpha}}$$ 
    for some $c>0$ depending on $C, \alpha$. Now $\mu(R_t)=0$, while $\nu$ has density $n$ in the base square $Q$, thus $\nu(R_t)\geq \frac{c}{n^{1/\alpha}}$.

    It remains to prove the third statement. Fix a sequence $n_1< n_2< \dots$ and put $N_k = \prod_{i=1}^{k}n_i$. We also fix permutations $\tau_k \in \Sym(n_k)$ with $\ds(\lambda, \widehat{\tau_k})\leq C\frac{\log n_k}{n_k}$, guaranteed to exist by Theorem \ref{thm:halton} and Lemma \ref{lemma:ds_permutation_regularization}. We build a sequence of permutations as follows: we start with $\pi_1=\tau_1\in \Sym(N_1)$, and if $\pi_{k-1}\in \Sym(N_{k-1})$ is given, we put
    $$\pi_k = \pi_{k-1}[\tau_k, \tau_k, \dots, \tau_k]\in \Sym(N_k),$$
    accounting for $N_{k-1}$ distinct copies of $\tau_k$ arranged according to $\pi_{k-1}$.
    Then $\supp \widehat{\pi_k}\subset \supp \widehat{\pi_{k-1}}$, and $\bigcap_{k=1}^{\infty}\supp \widehat{\pi_k}$ is a graph of an almost everywhere defined measure preserving bijection $f$, for which $\mu_f$ is the weak limit of $\widehat{\pi_k}$.

    \begin{claim} \label{claim:controlled_growth_holder}
        For $\alpha<1$, if $n_1< n_2< \dots$ are such that $N_{k+1}^{\alpha}\leq \beta N_k$ for some $\beta>0$ and every $k$ then $f$ is Hölder-$\alpha$ in a dense set of points with Hausdorff dimension one.
    \end{claim}

    \begin{proof}
        We will show that using the notation of Lemma \ref{lemma:dimension}, $f$ is Hölder-$\alpha$ at points of $F_{k_0, r}$. As the union of all these sets is dense and Lemma \ref{lemma:dimension} guarantees the clause on the dimension, this is sufficient.

        Fix $k_0, r$ and let $x\in F_{k_0, r}$. It suffices to check that for $x'\neq x$ near enough to $x$, $|f(x)-f(x')|\leq C_1|x-x'|^{\alpha}$ for some constant $C_1$. Now we can pick $I_k\supseteq I_{k+1}\supseteq I_{k+2}$ such that $I_j\in \mathcal{I}_{j, r}$ for $j=k, k+1, k+2$, and $x'\in I_{k}\setminus I_{k+1}$, while $x\in I_{k+2}$. Then
         $$|x-x'|\geq r|I_{k+1}|=\frac{r}{N_{k+1}}.$$
        On the other hand, $f(x), f(x')\in f(I_k)$ and $f(I_k)$ is an interval of length $\frac{1}{N_k}$, hence
        $$|f(x)-f(x')|\leq \frac{1}{N_k}.$$
        Comparing these bounds and inspecting the condition of the Claim \ref{claim:controlled_growth_holder} concerning the growth of $N_k$, this concludes the proof of Claim \ref{claim:controlled_growth_holder}.
    \end{proof}

    \begin{claim} \label{claim:fast_growth_well_approximable}
        If $n_1< n_2< \dots$ are such that $\beta' N_k\leq N_{k+1}^{\alpha}\leq \beta N_k$ for some $\beta, \beta'>0$ and every $k$ then $\ds(\mu_f, \widehat{\pi_k})=O\left(\frac{\log n}{n^{1/\alpha}}\right)$ whenever $n=N_k$.
    \end{claim}

    \begin{proof}
        We will analyze $\ds(\widehat{\pi_{k}}, \widehat{\pi_{k+1}})$.
        An axis-aligned rectangle $R$ evades or fully contains every base square of $\pi_k$, except for at most 4 ones. Thus $\ds(\widehat{\pi_{k}}, \widehat{\pi_{k+1}})$ can be bounded from above by four times their distance restricted to such a base square. However, restricted to any base square $\widehat{\pi_k}$ is simply a rescaled copy of the Lebesgue measure while $\widehat{\pi_{k+1}}$ is a rescaled copy of $\widehat{\tau_{k+1}}$. The scaling factor is $1/N_k$, thus by the choice of $\tau_{k+1}$, this implies that
        $$\ds(\widehat{\pi_{k}}, \widehat{\pi_{k+1}})\leq 4C\frac{\log n_{k+1}}{N_kn_{k+1}}\leq \frac{4C\log n_{k+1}}{N_{k+1}}.$$
        By the growth condition on $N_k$, we can further bound this by
        $$4C\frac{\log \beta N_k^{1/\alpha-1}}{(\beta' N_k)^{1/\alpha}}$$
        Then by triangle inequality:
        $$\ds(\widehat{\pi_k}, \mu_f)\leq 4C\sum_{i=k}^{\infty}\frac{\log \beta N_i^{1/\alpha-1}}{(\beta' N_i)^{1/\alpha}}=O\left(\frac{\log N_k}{N_k^{1/\alpha}}\right),$$
        by the superexponential growth of $N_k$ it is standard to check that the first term dominates this series. This concludes the proof of Claim \ref{claim:fast_growth_well_approximable}.
    \end{proof}

    Choosing $n_1, n_2, \dots$ satisfying the conditions of Claims \ref{claim:controlled_growth_holder}-\ref{claim:fast_growth_well_approximable} concludes the proof of the third statement of Theorem \ref{thm:regular_not_approximable}.
\end{proof}

\begin{remark}
    Observe that the previous construction can indeed be tweaked to construct permutons for which along a subsequence, \(D_n\) decays faster than any prescribed decreasing sequence. This follows from using very quickly growing $N_k$ and adjusting the proof of Claim \ref{claim:fast_growth_well_approximable}.
\end{remark}

\section{Approximation of the Brownian separable permuton}\label{sec:brownian_separable}

The content of the following theorem is illustrated in Figure~\ref{fig:maazoun}, reproduced from \cite{Maazoun_2020}.

\begin{theorem}[{{\cite[Theorem~1.6]{Maazoun_2020}}}] \label{thm:maazoun}
	Let $(\Delta_0,\Delta_1,\Delta_2)$ be a random vector with distribution \(\mathrm{Dirichlet}(\tfrac12,\tfrac12,\tfrac12)\).
	Let $\mu_0, \mu_1, \mu_2$ be independent and distributed like $\mu^p$, and
	conditionally on $\mu_0$, let $(X_0,Y_0)$ be a random point of distribution $\mu_0$.
	Let $\beta$ be an independent Bernoulli random variable with parameter $p$.
	We define the piecewise affine maps of the unit square into itself:
	\begin{equation}\begin{aligned}\label{BP_eq:definitiontheta}
	& \theta_0(x,y) & = (\eta_0(x), \zeta_0(y))& = \Delta_0(x,y) + (1-\Delta_0)(\mathbf{1}_{[x>X_0]},\mathbf{1}_{[y>Y_0]})  \\
	& \theta_1(x,y) & = (\eta_1(x), \zeta_1(y))&= \Delta_1(x,y) +\Delta_0(X_0,Y_0) +\Delta_2(0,\beta)              \\
	& \theta_2(x,y) & = (\eta_2(x), \zeta_2(y)) &= \Delta_2(x,y) +\Delta_0(X_0,Y_0) +\Delta_1(1,1-\beta)           
	\end{aligned}\end{equation}
	Then 
	\begin{equation}\Delta_0 (\theta_0{})_*\mu_0 + \Delta_1 (\theta_1{})_*\mu_1 + \Delta_2 (\theta_2{})_*\mu_2 \stackrel d= \mu^p,
	\label{BP_eq:ss}
	\end{equation}
\end{theorem}

\begin{figure}[htb]
	\centering
	
	\includegraphics{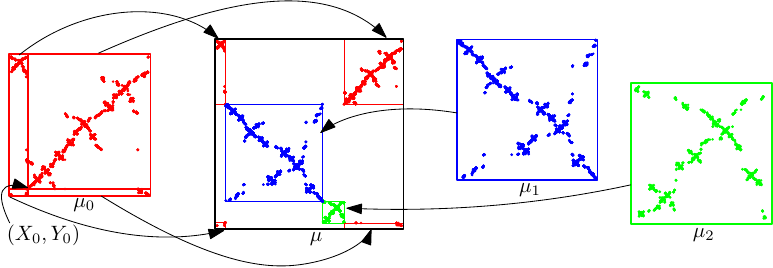}
	\caption{The construction of $\mu$ from three independent permutons distributed like $\mu$. Here $\beta = 1$ and $(\Delta_0,\Delta_1,\Delta_2) \approx (0.4,0.5,0.1)$. (Reproduced from \cite[Figure~2]{Maazoun_2020}; the original caption states \(\beta=0\), which appears to be a typo.)}  \label{fig:maazoun}
\end{figure}

We rely on this decomposition to prove Lemma~\ref{lemma:brownian_lip}, which in turn implies Theorem~\ref{thm:brownian_approx}.

\begin{proof}[Proof of Lemma \ref{lemma:brownian_lip}]
    Denote the unit square by $Q_{\emptyset}$. By Theorem \ref{thm:maazoun}, we see that there are two squares $Q_1, Q_2$ inside $Q_{\emptyset}$, such that $\mu^p(Q_i) = \lambda(\Pr_x Q_i)=\lambda(\Pr_y Q_i)$. (The blue and green squares correspond to these in Figure \ref{fig:maazoun}.) Inside these squares we can find two squares with the same property, denoted by $Q_{i, 1}$, $Q_{i, 2}$, $i=1, 2$. Proceeding recursively, we can define $Q_{i_1, i_2, \dots, i_n}$ for any $n$, and for any $\mathbf{i}=i_1, i_2, \dots$, the intersection $\bigcap_{n} Q_{i_1, \dots, i_n}$ almost surely consists of a single point $(x_{\mathbf{i}}, y_{\mathbf{i}})$. This will be a point of the graph of the function $f$ supporting $\mu^p$. We claim that with probability one some $\mathbf{i}$ will determine a point at which $f$ is Lipschitz.

    To simplify notation, let $Q'\subseteq Q$ be two consecutive squares in this sequence for some $\mathbf{i}$. Let their projections to the $x$-axis be $[a', b']$ and $[a, b]$. We say that $Q'$ is $r$-good for some $0<r<1/2$ if
    \begin{enumerate}[before=\leavevmode,label=\upshape(\Roman*)]
        \item\label{cond1} $a'>a+r(b-a)$, $b'<b-r(b-a)$,
        \item\label{cond2} $b'-a'>r(b-a)$,
    \end{enumerate}
    that is, \(Q'\) is not too small and its \(x\)-projection stays a positive distance away from both endpoints of the \(x\)-projection of \(Q\). We say that $\mathbf{i}$ is $(r, n_0)$-good if any $Q_{i_1, ..., i_n}$ is $r$-good for $n\geq n_0$. 

    We make the following claims.
    
    \begin{claim} \label{claim:good_implies_lip}
        If some $\mathbf{i}$ is $(r, n_0)$-good for some $r, n_0>0$, then $(x_{\mathbf{i}}, y_{\mathbf{i}})$ is well-defined and the supporting function \(f\) is Lipschitz at \(x_{\mathbf{i}}\).
    \end{claim}

    \begin{claim}\label{claim:exists_good}
        If $r$ is small enough then there exists an $(r, n_0)$-good $\mathbf{i}$ almost surely.
    \end{claim}

    These claims combined prove the lemma, thus the rest of the proof is devoted to verifying them.

    \begin{proof}[Proof of Claim \ref{claim:good_implies_lip}]
        The fact that \((x_{\mathbf i},y_{\mathbf i})\) is well-defined follows immediately from property~\ref{cond1}. To simplify notation, let $Q_n = Q_{i_1, \dots, i_n}$, and let $\Pr_x Q_n = [a_n, b_n]$.
        
        Fix $x\neq x_{\mathbf{i}}$ in $[0, 1]$. Then there exists $n\geq n_0$ such that $x\in [a_n, b_n]$, but $x \notin [a_{n+1}, b_{n+1}]$. We would like to quantify the distance between these points by giving a good lower bound. From the relative positions of these intervals, we have
        $$|x-x_{\mathbf{i}}|> \min(|x_{\mathbf{i}}-a_{n+1}|, |x_{\mathbf{i}}-b_{n+1}|)>  \min(|a_{n+2}-a_{n+1}|, |b_{n+2}-b_{n+1}|).$$
        Due to symmetry, it suffices to bound one of these differences. First applying \ref{cond1}, and then \ref{cond2}, we find
        $$a_{n+2}-a_{n+1}> r(b_{n+1}-a_{n+1})>r^2(b_n-a_n),$$
        hence 
        $$|x-x_{\mathbf{i}}|> r^2(b_n-a_n).$$
        On the other hand, for any $(x,y)\in \supp \mu^p$, as both $y$ and $y_{\mathbf{i}}$ are in $Q_n$, we have $|y-y_{\mathbf{i}}|\leq |a_n-b_n|$. These two bounds combined prove that the supporting function \(f\) is Lipschitz at \(x_{\mathbf{i}}\), verifying the claim.
    \end{proof}

    \begin{proof}[Proof of Claim \ref{claim:exists_good}]
        Notice that for any initial segment $i_1, \dots, i_{n_0}$, the $(r, n_0)$-good $\mathbf{i}$s are in bijection with the boundary of a Galton--Watson $GW(r)$ tree defined by the following offspring distribution:
        after sampling $(\Delta_0,\Delta_1,\Delta_2)$ according to $ \mathrm{Dirichlet}(\tfrac 1 2,\tfrac 1 2,\tfrac 1 2)$, we sample $X_0\sim \mathrm{Unif}(0, \Delta_0)$. If $X_0>r$, then if one of $\Delta_1, \Delta_2$ is larger than $r$, then there is 1 offspring, while if both are, then there are 2 offspring. If $X_0<r$, then there are no offspring. Then for the random number $Z(r)$ of offspring, $Z(r)\to 2$ in probability as $r\to 0$, hence $GW(r)$ survives with positive probability, yielding a non-empty boundary with positive probability $p_r$. Taking countably many pairwise disjoint initial segments (such as $1, 21, 221, \dots$) shows that there exists an $(r, n_0)$-good $\mathbf{i}$ for some $n_0$ if any of the countably many independent events with probability $p_r$ occurs, which clearly happens almost surely.   
    \end{proof}

\end{proof}

\section*{Acknowledgements}

The author is grateful to Balázs Ráth for communicating the problem and for many helpful and interesting discussions. The author also thanks Miklós Abért, István Berkes, Joanna Jasińska, and Bálint Virág for helpful discussions.

\printbibliography

\end{document}